\documentclass[11pt,letterpaper]{amsart}
\usepackage{cases}
\usepackage{mathrsfs}
\usepackage{color}
\usepackage{latexsym,amssymb}
\usepackage{a4wide}
\usepackage{amsthm}
\usepackage{amsmath}
\usepackage{graphicx}
\input xy
\xyoption{all}
\usepackage{verbatim}
\usepackage[lite,abbrev,msc-links]{amsrefs}

\numberwithin{equation}{section}
\begin{document}
	\theoremstyle{plain}
	\newtheorem{thm}{Theorem}[section]
	\newtheorem{lem}[thm]{Lemma}
	\newtheorem{cor}[thm]{Corollary}
	\newtheorem{cor*}[thm]{Corollary*}
	\newtheorem{prop}[thm]{Proposition}
	\newtheorem{prop*}[thm]{Proposition*}
	\newtheorem{conj}[thm]{Conjecture}
	\theoremstyle{definition}
	\newtheorem{construction}{Construction}
	\newtheorem{notations}[thm]{Notations}
	\newtheorem{question}[thm]{Question}
	\newtheorem{prob}[thm]{Problem}
	\newtheorem{rmk}[thm]{Remark}
	\newtheorem{remarks}[thm]{Remarks}
	\newtheorem{defn}[thm]{Definition}
	\newtheorem{claim}[thm]{Claim}
	\newtheorem{assumption}[thm]{Assumption}
	\newtheorem{assumptions}[thm]{Assumptions}
	\newtheorem{properties}[thm]{Properties}
	\newtheorem{exmp}[thm]{Example}
	\newtheorem{comments}[thm]{Comments}
	\newtheorem{blank}[thm]{}
	\newtheorem{observation}[thm]{Observation}
	\newtheorem{defn-thm}[thm]{Definition-Theorem}
	\newtheorem*{Setting}{Setting}

	\newcommand{\sA}{\mathscr{A}}
	\newcommand{\sB}{\mathscr{B}}
	\newcommand{\sC}{\mathscr{C}}
	\newcommand{\sD}{\mathscr{D}}
	\newcommand{\sE}{\mathscr{E}}
	\newcommand{\sF}{\mathscr{F}}
	\newcommand{\sG}{\mathscr{G}}
	\newcommand{\sH}{\mathscr{H}}
	\newcommand{\sI}{\mathscr{I}}
	\newcommand{\sJ}{\mathscr{J}}
	\newcommand{\sK}{\mathscr{K}}
	\newcommand{\sL}{\mathscr{L}}
	\newcommand{\sM}{\mathscr{M}}
	\newcommand{\sN}{\mathscr{N}}
	\newcommand{\sO}{\mathscr{O}}
	\newcommand{\sP}{\mathscr{P}}
	\newcommand{\sQ}{\mathscr{Q}}
	\newcommand{\sR}{\mathscr{R}}
	\newcommand{\sS}{\mathscr{S}}
	\newcommand{\sT}{\mathscr{T}}
	\newcommand{\sU}{\mathscr{U}}
	\newcommand{\sV}{\mathscr{V}}
	\newcommand{\sW}{\mathscr{W}}
	\newcommand{\sX}{\mathscr{X}}
	\newcommand{\sY}{\mathscr{Y}}
	\newcommand{\sZ}{\mathscr{Z}}
	\newcommand{\bZ}{\mathbb{Z}}
	\newcommand{\bN}{\mathbb{N}}
	\newcommand{\bQ}{\mathbb{Q}}
	\newcommand{\bC}{\mathbb{C}}
	\newcommand{\bR}{\mathbb{R}}
	\newcommand{\bH}{\mathbb{H}}
	\newcommand{\bD}{\mathbb{D}}
	\newcommand{\bE}{\mathbb{E}}
	\newcommand{\bV}{\mathbb{V}}
	\newcommand{\cV}{\mathcal{V}}
	\newcommand{\cF}{\mathcal{F}}
	\newcommand{\bfM}{\mathbf{M}}
	\newcommand{\bfN}{\mathbf{N}}
	\newcommand{\bfX}{\mathbf{X}}
	\newcommand{\bfY}{\mathbf{Y}}
	\newcommand{\spec}{\textrm{Spec}}
	\newcommand{\dbar}{\bar{\partial}}
	\newcommand{\ddbar}{\partial\bar{\partial}}
	\newcommand{\redref}{{\color{red}ref}}
	
	\title[] {Minimal extension property of direct images}

	\author[Chen Zhao]{Chen Zhao}
	\email{czhao@ustc.edu.cn}
	\address{School of Mathematical Sciences,
		University of Science and Technology of China, Hefei, 230026, China}

	\begin{abstract}
Given a projective morphism $f:X\to Y$ from a complex space to a complex manifold, we prove the Griffiths semi-positivity and minimal extension property of the direct image sheaf $f_\ast(\sF)$. Here, $\sF$ is a coherent sheaf on $X$, which consists of the Grauert-Riemenschneider dualizing sheaf, a multiplier ideal sheaf, and a variation of Hodge structure (or more generally, a tame harmonic bundle). 
	\end{abstract}
	
	\maketitle
	
\section{introduction}
Given a holomorphic map between projective manifolds $f:X\to Y$ and a holomorphic line bundle $L$ endowed with a singular Hermitian metric $h$, the positivity of $f_\ast(\omega_{X/Y}\otimes L\otimes\sI(h))$ has been a topic of great interest in decades. This positivity problem plays crucial roles in many subjects in complex algebraic geometry such as Iitaka $C_{n,m}$ conjecture \cite{Kawamata1985,Viehweg1983,CP2017} and the moduli of projective varieties \cite{Viehweg1995}.
 
 A significant breakthrough in recent years has been the recognition of Nakano semi-positivity (also known as the minimal extension property) of $f_\ast(\omega_{X/Y}\otimes L\otimes\sI(h))$, attributed to Berndtsson \cite{Berdtsson2009}, Paun-Takayama \cite{PT2018}, Hacon-Popa-Schnell \cite{HPS2018}. This observation has enabled Cao-Paun \cite{CP2017} to solve the Iitaka $C_{n,m}$ conjecture over abelian varieties (see also \cite{HPS2018}).
The purpose of this article is to demonstrate that the minimal extension property holds for a much broader class of direct images. In this case, the holomorphic line bundle $(L, h)$ could be replaced by degenerate bundles, such as a variation of Hodge structure or, a tame harmonic bundle in a broader context (corresponding to certain parabolic Higgs bundles, see Simpson \cite{Simpson1988, Simpson1990} and Mochizuki \cite{Mochizuki20072, Mochizuki20071}).
\subsection{Main result}
Let $f:X\rightarrow Y$ be a projective surjective morphism from a complex space $X$\footnote{All complex space is assumed to be reduced and irreducible.} to a complex manifold $Y$. Let $X^o\subset X$ be a dense Zariski open subset and $(E,h)$ be a holomorphic vector bundle endowed with a singular Hermitian metric (Definition \ref{defn_shmonvb}). 
Let $S_X(E,h)$ be the sheaf on $X$ defined as follows. Let $U\subset X$ be an open subset. The space $S_X(E,h)(U)$ consists of holomorphic $E$-valued $(n,0)$-forms $\alpha$ on $U\cap X^o$ such that 
$\{\alpha,\alpha\}$ is locally integrable near every point of $U$. 
We define $S_{X/Y}(E,h)$ as $S_X(E,h)\otimes f^\ast(\omega_Y^{-1})$. 
\begin{thm}\label{thm_main}
If the holomorphic vector bundle $(E,h)$ is tame on $X$ (Definition \ref{defn_tame_Hermitian_bundle}) and $\Theta_h(E)\geq_{\rm Nak}^s 0$ on $X^o$ (Definition \ref{defn_singularNakanopositive}), then $S_{X/Y}(E,h)$ is a coherent sheaf on $X$ and the direct image sheaf 
		$$\sF=f_\ast (S_{X/Y}(E,h))$$
		has a canonical singular Hermitian metric which is Griffiths semi-positive and satisfies the minimal extension property.
	\end{thm}
\begin{rmk}
	We adopt Caltaldo's concept of Nakano semi-positivity for singular Hermitian metrics (see Definition \ref{defn_singularNakanopositive}) because it allows for the validity of H\"ormander's estimate (\cite[Proposition 4.1.1]{CataldoAndrea1998}, also see  \cite[Th\'eor\`eme 5.1]{Demailly1982}) and the optimal Ohsawa-Takegoshi extension theorem (Guan-Zhou \cite{Guan-Zhou2015}, Guan-Mi-Yuan \cite{GMY2023}). Defining "Nakano semi-positivity" in a way that does not rely on approximations using \( C^2 \) metrics, while still ensuring H\"ormander's estimate and the optimal Ohsawa-Takegoshi extension theorem, presents an intriguing challenge. Relevant works addressing this include \cites{KP2021,DNW2021,DNWZ2023,DNZZ2024,In2022,PT2018,Rau2015}.
\end{rmk}
\begin{rmk}
	Since we are interested in the case when the vector bundle $E$ arises from a variation of Hodge structure or a tame harmonic bundle, we do not require $E$ to have a holomorphic extension to $X$. It is possible that there will be some fiber of $f$ where $E$ is nowhere defined, and $h$ may not extend to $X$. This presents the main difficulty in this article compared to known works. 
	The primary contribution of this article is the introduction of the "tame" condition (Definition \ref{defn_tame_Hermitian_bundle}), which is motivated by the theory of degeneration of Hodge structure (Schmid \cite{Schmid1973}, Cattani-Kaplan-Schmid \cite{Cattani_Kaplan_Schmid1986}) and the theory of tame harmonic bundles (Simpson \cite{Simpson1990}, Mochizuki \cite{Mochizuki20072, Mochizuki20071}). Roughly speaking, the tameness of the Hermitian vector bundle $(E,h)$ means that the dual metric $h^\ast$ has at most polynomial growth at every point on $X$. This condition allows for the use of techniques used in Paun-Takayama \cite{PT2018} and Hacon-Popa-Schnell \cite{HPS2018} on the degenerate loci of $(E,h)$. 
We want to point out that if $(E,h)$ is Nakano semi-positive on $X^o$, then it is also tame on $X^o$. Therefore, the main concern of the tameness condition in Theorem \ref{thm_main} is the asymptotic behavior of the metric $h$ on the boundary $X\backslash X^o$.
\end{rmk}
\begin{rmk}
	The construction of $S_X(E,h)$ was  introduced in \cite{SZ2022} and \cite{SC2021}. It offers a convenient way to combine Hodge-theoretic objects, like the Koll\'ar-Saito $S$-sheaf, with transcendental objects, such as the multiplier ideal sheaf. Typically, under certain conditions, $S_X(E,h)$ exhibits good Hodge-theoretic properties, such as  Koll\'ar's package (see \cite{SZ2022}), as well as good transcendental properties, including the strong openness property and the Ohsawa-Takegoshi extension property (see \cite{SC2021}).
\end{rmk}
\subsection{Example: multiplier ideal sheaf}
The first example is the case when the bundle $E$ do not degenerate. When $X=X^o$ ($X$ is smooth in particular), $E$ is a holomorphic line bundle with a singular Hermitian metric of semi-positive curvature, the aforementioned theorem implies the positivity result of the direct image sheaf $f_\ast (\omega_{X/Y}\otimes E\otimes \sI(h))$, as proven by Paun-Takayama \cite{PT2018} and Hacon-Popa-Schnell \cite[Theorem 21.1]{HPS2018}. If $E$ is of higher rank and $h$ is a metric satisfying conditions in Theorem \ref{thm_main}, then using Caltaldo's notation $E(h)$ (consisting of locally $L^2$ holomorphic sections in $E$), $f_\ast (\omega_{X/Y}\otimes E(h))$ has a canonical singular Hermitian metric which is Griffiths semi-positive and satisfies the minimal extension property.
\subsection{Example: multiplier $S$-sheaf}\label{section_exp_S_sheaf}
Let $\mathbb{V}$ be a variation of Hodge structure on a regular Zariski open subset of a projective variety $X$. Koll\'ar \cite{Kollar1986_2} introduced a coherent sheaf $S(IC_X(\mathbb{V}))$ that generalizes the dualizing sheaf. He conjectured that $S(IC_X(\mathbb{V}))$ satisfies Koll\'ar's package (including the torsion-freeness, the  injectivity theorem, Koll\'ar's vanishing theorem and the decomposition theorem). This conjecture was subsequently proven by Saito \cite{MSaito1991} using the theory of mixed Hodge modules. Saito's proof is based on the observation that $S(IC_X(\mathbb{V}))$ represents the highest index Hodge component of the intermediate extension $IC_X(\mathbb{V})$ as a Hodge module. In \cite{SZ2022}, the authors provide a new proof of Koll\'ar's conjecture using the $L^2$-method. This is based on their observation that $S(IC_X(\mathbb{V}))$ is isomorphic to some $S_X(E,h)$ for certain Hermitian bundle $(E,h)$ arising from the variation of Hodge structure (see below).
The $S$-sheaf has played a crucial role in the application of Hodge module theory to complex algebraic geometry (see \cite{Popa2018} for a comprehensive survey).

Let $X$ be a complex space and $X^o \subset X_{\rm reg}$ a Zariski open subset. Let $\bV := (\cV, \nabla, \cF^\bullet, Q)$ be a polarized complex variation of Hodge structure on $X^o$. To establish the Nadel vanishing theorem for $S(IC_X(\bV))$, the authors of \cite{SC2021} introduce a multiplier $S$-sheaf, denoted as $S(IC_X(\bV),\varphi)$,  a combination of the $S$-sheaf $S(IC_X(\bV))$ and the multiplier ideal sheaf associated with a quasi-plurisubharmonic (quasi-psh) function $\varphi:X\to[-\infty,\infty)$.
Let $h_Q$ be the Hodge metric defined as $(u,v)_{h_Q} := Q(Cu, \overline{v})$, where $Q$ is the polarization of $\bV$ and $C$ is the Weil operator. The multiplier $S$-sheaf is then defined as
\[
S(IC_X(\bV),\varphi) := S_X(S(\bV), e^{-\varphi}h_Q),
\]
where $S(\bV) := \cF^{\max\{k|\cF^k\neq0\}}$ is the top indexed nonzero piece of the Hodge filtration $\cF^\bullet$. This type of sheaf possesses several good properties, such as the strong openness property and the Ohsawa-Takegoshi extension property. Additionally, $S(IC_X(\bV),0) = S(IC_X(\bV))$. For more details on the multiplier $S$-sheaf and its relation to $S(IC_X(\bV))$ and $\sI(\varphi)$, readers may refer to \cite{SC2021}.

Let $S_{X/Y}(IC_X(\bV),\varphi):=S(IC_X(\bV),\varphi)\otimes f^\ast(\omega_Y^{-1})$. As a consequence of Theorem \ref{thm_Higgs}, we obtain the following.
\begin{thm}\label{thm_multi_S}
	Let $(E,e^{-\varphi}h)$ be a Hermitian vector bundle on $X$ such that $\varphi$ is a quasi-psh function on $X$ and $h$ is a smooth metric. Assume that $\sqrt{-1}\ddbar(\varphi)+\sqrt{-1}\Theta_h(E)\geq_{\rm Nak}0$. Let $f:X\to Y$ be a surjective projective morphism to a complex manifold $Y$. Then 
	$$f_\ast\left(S_{X/Y}(IC_X(\bV),\varphi)\otimes E\right)$$
	has a singular Hermitian metric which is Griffiths semi-positive and satisfies the minimal extension property. 
\end{thm}
In particular, if $E$ is a holomorphic vector bundle on $X$ endowed with a smooth Hermitian metric with Nakano semi-positive curvature (i.e., $\varphi=0$ in Theorem \ref{thm_multi_S}), then $f_\ast\left(S_{X/Y}(IC_X(\bV))\otimes E\right)$ has a singular Hermitian metric that is Griffiths semi-positive and satisfies the minimal extension property. This generalizes the result of Schnell-Yang \cite{SY2023} to the relative case.
\subsection{Example: parabolic Higgs bundle}\label{section_exp_Higgs}
The concept of the multiplier $S$-sheaf, as shown in the previous example, can be extended to the framework of non-abelian Hodge theory. This extension of the multiplier $S$-sheaf is elaborated on in \cite{SZ2022}.

Let $X$ be a smooth, projective variety, and let $D$ be a reduced simple normal crossing divisor on $X$. Let's consider a locally abelian parabolic Higgs bundle $(H,\{{_E}H\}_{E\in{\rm Div}_D(X)},\theta)$ on $(X,D)$. This bundle consists of the following data: 
\begin{itemize}
	\item A locally abelian parabolic vector bundle $(H,\{{_E}H\}_{E\in{\rm Div}_D(X)})$ with parabolic structures on $D$. Here, the filtration $\{{_E}H\}$ is indexed by the set ${\rm Div}_D(X)$, which consists of $\bR$-divisors whose support lies in $D$.
	\item A Higgs field $\theta:H|_{X\backslash D}\to H|_{X\backslash D}\otimes \Omega_{X\backslash D}$, which has regular singularity along $D$. 
\end{itemize}
This parabolic Higgs bundle is required to have vanishing parabolic Chern classes and to be polystable with respect to an ample line bundle $A$ on $X$. 

The main focus of this study is to examine a specific extension, denoted as $P_{E,(2)}(H)$, of $H|_{X\backslash D}$. To define this extension, let $E$ be an $\mathbb{R}$-divisor supported on $D$. We denote ${_{<E}}H$ as $\cup_{E'<E}{_{E'}}H$.
The coherent sheaf $P_{E,(2)}(H)$ is determined by the following conditions.
\begin{enumerate}
	\item ${_{<E}}H\subset P_{E,(2)}(H)\subset {_{E}}H$. 
	\item Let $x$ be a point in $D$ and $(U;z_1,\dots, z_n)$ be holomorphic local coordinates on some open neighborhood $U$ of $x$ in $X$, such that $D=\{z_1\cdots z_r=0\}$. We denote $D_i=\{z_i=0\}$ for $i=1,\dots,r$.  Now, let $\{W_{m,i}({_{E}}H)\}_{m\in\mathbb{Z}}$ be the monodromy weight filtration on ${_{E}}H|_U$ at $x$, with respect to the nilpotent part of the residue map ${\rm Res}_{D_i}(\theta)$ of the Higgs field along $D_i$. Then, we have:
	\begin{align}
		P_{E,(2)}(H)|_U={_{<E}}H+\bigcap_{i=1}^rW_{-2,i}({_{E}}H).
	\end{align}
\end{enumerate}
When $E=D$, $P_{D,(2)}(H)$ represents the sheaf of $L^2$-holomorphic sections with coefficients in $H$. This construction was originally introduced by S. Zucker \cite{Zucker1979} for algebraic curves, and it involves $H$ arising from a variation of Hodge structure. Consequently, it is a significant subject of study in the context of $L^2$-cohomology of a variation of Hodge structure.

To extend Zucker's construction \cite{Zucker1979} to higher-dimensional bases and non-canonical indexed extensions, we introduce $P_{E,(2)}(H)$. Specifically, when $E\geq0$, $P_{D-E,(2)}(H)$ combines elements from both $P_{D,(2)}(H)$ and the multiplier ideal sheaf associated with $E$. This aspect makes $P_{E,(2)}(H)$ more convenient in applications where $E\neq D$. It can be proven that $P_{E,(2)}(H)$ is always locally free.

According to the non-abelian Hodge theory of Simpson \cite{Simpson1988, Simpson1990} and Mochizuki \cite{Mochizuki2006, Mochizuki20071}, a $\mu_A$-polystable regular parabolic flat bundle $(V, \{{_E}V\}_{E\in{\rm Div}_D(X)}, \nabla)$ is associated with the parabolic Higgs bundle $(H,\{{_E}H\}_{E\in{\rm Div}_D(X)},\theta)$. Moreover, there exists an isomorphism between the $C^\infty$ complex bundles:
$$\rho:H|_{X\backslash D}\otimes_{\sO_{X\backslash D}}\sC^\infty_{X\backslash D}=V|_{X\backslash D}\otimes_{\sO_{X\backslash D}}\sC^\infty_{X\backslash D}.$$
In particular, the $C^\infty$ complex bundle associated with $H|_{X\backslash D}$ has two complex structures: $\dbar$, the complex structure of the Higgs bundle $H|_{X\backslash D}$, and $\nabla^{0,1}$, the $(0,1)$-part of $\nabla$ in the flat bundle $V|_{X\backslash D}$.

In this setting, Theorem \ref{thm_main} implies the following result.
\begin{thm}\label{thm_Higgs}
	Let $K$ be a locally free subsheaf of $H|_{X\backslash D}$ satisfying the following conditions:
	\begin{itemize}
		\item Holomorphicity: $\nabla^{0,1}(K) = 0$, meaning that $K$ is holomorphic with respect to both the complex structures $\dbar$ and $\nabla^{0,1}$.
		\item Weak transversality\footnote{This condition is referred to as weak transversality due to Griffiths's transversality when $H$ arises from a variation of Hodge structure with $\{F^p\}_{p\in\bZ}$ as the Hodge filtration and $K = F^p$ for some $p$.}: $(\nabla - \theta)(K) \subset K \otimes \sA^{1,0}_{X\backslash D}$. 
	\end{itemize}
Let $L$ be a line bundle on $X$ such that $L\simeq_{\mathbb{R}}B+N$, where $B$ is a semi-positive $\mathbb{R}$-divisor and $N$ is an $\mathbb{R}$-divisor on $X$ supported on $D$. Let $F$ be a Nakano semi-positive vector bundle on $X$. Let $j:X\backslash D\to X$ be the immersion, and let $f:X\to Y$ be a surjective projective morphism to a complex manifold $Y$.

Then, the sheaf $f_\ast\left(\omega_{X/Y}\otimes (P_{D-N,(2)}(H)\cap j_\ast K)\otimes F\otimes L\right)$ has a singular Hermitian metric which is Griffiths semi-positive and satisfies the minimal extension property.
\end{thm}
In \cite{SZtwisted}, the authors demonstrate that $\omega_X\otimes (P_{D-N,(2)}(H)\cap j_\ast K)\otimes F\otimes L$ satisfies Koll\'ar's package. In particular, they establish that $f_\ast\left(\omega_{X/Y}\otimes (P_{D-N,(2)}(H)\cap j_\ast K)\otimes F\otimes L\right)$ is weakly positive in the sense of Viehweg. 
\begin{rmk}
	Let $(H,\{{_E}H\}_{E\in{\rm Div}_D(X)},\theta)$ be the parabolic Higgs bundle associated with a variation of Hodge structure and let $K=S(\bV)$. Let $\varphi_N$ be the weight quasi-psh function associated with the divisor $N$. Then $P{_{D-N,(2)}}(H)\cap j_\ast K$ is coincide with the multiplier $S$-sheaf $S(IC_X(\bV),\varphi_N)$.
\end{rmk}
\subsection{Example: parabolic bundle}
Let $X$ be a smooth projective variety and $D\subset X$ be a simple normal crossing divisor on $X$. Let $(H,\{{_E}H\}_{E\in{\rm Div}_D(X)})$ be a locally abelian parabolic bundle on $(X,D)$ with vanishing parabolic Chern classes. This bundle is also polystable with respect to an ample line bundle $A$ on $X$.
In this case, we can consider $(H,\{{_E}H\}_{E\in{\rm Div}_D(X)})$ as a parabolic Higgs bundle with a vanishing Higgs field. Consequently, $P_{E,(2)}(H)={_{<E}}H$. By selecting $K=H|_{X\backslash D}$ in Theorem \ref{thm_main}, the conditions of holomorphicity and weak transversality are satisfied for $K$. Therefore, Theorem \ref{thm_Higgs} implies the following.
\begin{thm}
	Let $X$ be a smooth, projective variety and $D$ a simple normal crossing divisor on $X$. Let $(H,\{{_E}H\}_{E\in{\rm Div}_D(X)})$ be a locally abelian parabolic bundle on $(X,D)$ with vanishing parabolic Chern classes, which is polystable with respect to an ample line bundle $A$ on $X$. Let $L$ be a line bundle on $X$ such that $L\simeq_{\bR}B+N$, where $B$ is a semi-positive $\bR$-divisor and $N$ is an $\bR$-divisor on $X$ supported on $D$. Let $F$ be an arbitrary Nakano semi-positive vector bundle on $X$. 
	
	Let $f:X\to Y$ be a surjective projective morphism to a complex manifold $Y$.
	Then the direct image sheaf $f_\ast\left(\omega_{X/Y}\otimes {_{<D-N}}H\otimes F\otimes L\right)$ has a singular Hermitian metric which is Griffiths semi-positive and satisfies the minimal extension property.
\end{thm}
This article is structured as follows. In Section 2, we provide a review of basic concepts such as a singular metric on a torsion-free sheaf, Caltaldo's notion of Nakano semi-positivity, Hacon-Popa-Schnell's notion of minimal extension property, and Guan-Mi-Yuan's optimal Ohsawa-Takegoshi extension theorem. Section 3 introduces and examines $S_X(E,h)$, while also explaining its connection to significant transcendental and Hodge theoretic objects. The main result is demonstrated in Section 4.
\section{preliminary}
\subsection{Positivity of singular Hermitian metrics on a torsion free coherent sheaf}
	In this subsection, we review the notion of Nakano/Griffiths positivity for singular Hermitian metrics.
	
	Throughout this subsection, let $X$ be a complex manifold of dimension $n$ and $E$ be
	a vector bundle of rank $r$ on $X$. 
\begin{defn}[Nakano positivity and Griffiths positivity]
	A $C^2$ smooth Hermitian metric $h$ on $E$ defines the Chern curvature and associated Hermitian form:
	\begin{align}
		\sqrt{-1}\Theta_h\in C^0(X,\Lambda^{1,1}\otimes{\rm End}(E))\quad\textrm{and}\quad 	\sqrt{-1}\tilde{\Theta}_h\in C^0(X,{\rm Herm}(T_X\otimes E)).
	\end{align}	
		The metric $h$ is said to be
		\begin{itemize}
			\item \emph{Nakano semi-positive}, denoted as $\Theta_h(E)\geq_{\rm Nak}0$, if
			for every $u\in T_X\otimes E$, it holds that
			$\sqrt{-1}\tilde{\Theta}_h(u,u)\geq 0$.
			\item \emph{Griffiths semi-positive}, 
			 if for all $\xi\in TX$ and $s\in E$, it holds that 
			$\sqrt{-1}\tilde{\Theta}_h(\xi\otimes s,\xi\otimes s)\geq 0$.
		\end{itemize} 
\end{defn}

We review the singular version of the Griffiths positivity and  Nakano positivity of  a singular hermitian metric on a vector bundle.
First, the concept of a singular Hermitian metric, as introduced by Berndtsson-Paun \cite{BP2008}, Paun-Takayama \cite{PT2018}, and Hacon-Popa-Schnell \cite{HPS2018}, is defined as follows.
	\begin{defn}\label{defn_shmonvb}
		A \emph{singular Hermitian metric} on a vector bundle $E$ is a function $h$ that associates
		to every point $x\in X$ a singular hermitian inner product $|-|_{h,x} : E_x \rightarrow  [0, +\infty]$ on
		the complex vector space $E_x$, subject to the following two conditions:
		\begin{itemize}
			\item $h$ is finite and positive definite almost everywhere, meaning that for all $x$
			outside a set of measure zero, $|-|_{h,x}$ is a singular hermitian inner product on $E_x$.
			\item $h$ is measurable, meaning that the function
			$$|s|_h : U \rightarrow  [0, +\infty], x \mapsto|s(x)|_{h,x}$$
			is measurable whenever $U\subset X$ is open and $s\in H^0(U,E)$.
		\end{itemize}
	\end{defn}
\begin{defn}
	Let $\sF$ be a torsion free coherent sheaf on $X$. Denote by $X(\sF)\subset X$ the maximal open subset where $\sF$ is locally free and let $E:=\sF|_{X(\sF)}$.
	A \emph{singular Hermitian metric on  $\sF$} is a singular Hermitian metric $h$ on the holomorphic vector bundle $E$. 
\end{defn}	
	A singular Hermitian metric $h$ on $E$ induces a dual singular Hermitian metric $h^\ast$ on $E^\ast$.
	\begin{defn}
	A singular hermitian metric $h$ on a vector bundle $E$ is called \emph{Griffiths semi-positive} if $\log|u|^2_{h^\ast}$ is plurisubharmonic for any local holomorphic section $u$ of $E^\ast$. A singular hermitian metric $h$ on a torsion free coherent sheaf $\sF$ is called \emph{Griffiths semi-positive} if $(\sF|_{X(\sF)},h)$ is Griffiths semi-positive.
	\end{defn}
	When $h$ is a smooth Hermitian metric on $E$, the above definition coincides with the classical Griffiths positivity.

The following definition, which can be regarded as the singular version of Nakano positivity, was introduced by Cataldo \cite{CataldoAndrea1998} and Guan-Mi-Yuan \cite{GMY2023}.
\begin{defn}\label{defn_singularNakanopositive}
	 Let $\omega$ be a Hermitian form on $X$ and $\theta$ be a continuous real $(1,1)$-form on $X$. A singular Hermitian metric $h$ on $E$ is called \emph{$\theta$-Nakano semi-positive in the sense of approximations}, denoted by 
	 	$$\Theta_h(E) \geq_{\textrm{Nak}}^s \theta\otimes Id_E$$ if
	 	there is a collection of data $(\Sigma,X_j,h_{j,s})$ satisfying that 
	\begin{enumerate}
		\item $\Sigma\subset X$ is a closed
		set of measure zero;
		\item $\{X_j\}_{j=1}^{+\infty}$ is an open cover of $X$ made of  sequence of relatively compact subsets of $X$ such that $X_1\Subset X_2\Subset\cdots\Subset X_j\Subset X_{j+1}\Subset \cdots$;
		\item For each $X_j$, there exists a sequence of $C^2$ Hermitian metrics $\{h_{j,s}\}_{s=1}^{+\infty}$ on $X_j$ 
		such that
		\begin{align}
			\lim_{s\rightarrow +\infty} h_{j,s} = h \quad\textrm{point-wisely on}\quad X_j\setminus \Sigma
		\end{align}
	and for each $x\in X_j$ and $e\in E_x$ we have
	\begin{align}\label{align_metricincrease}
		|e|_{h_{j,s}}\nearrow |e|_{h_{j,s+1}} \quad\textrm{as}\quad s\nearrow \mathbb{N};
	\end{align}
\item For each $X_j$, there exists a sequence of continuous functions $\lambda_{j,s}$ on $X_j$ and a
continuous function $\lambda_j$ on $X_j$ subject to the following requirements:
\begin{itemize}
	\item $\Theta_{h_{j,s}}(E)\geq_{\rm Nak}\theta-\lambda_{j,s}\omega\otimes {\rm Id}_E$ on $X_j$;
	\item $\lambda_{j,s}\rightarrow 0$ almost everywhere on $X_j$;
	\item $0\leq \lambda_{j,s}\leq \lambda_j$ on $X_j$ for any $s\in\mathbb{N}$.
	\end{itemize}
	\end{enumerate} 
	Especially, when $\theta = 0$, the singular Hermitian metric $h$ is called \emph{singular Nakano
	semi-positive}, denoted by \(\Theta_h(E) \geq_{\rm Nak}^{ s} 0\). 
\end{defn}

\subsection{Ohsawa-Takegoshi extension theorem and the minimal extension property}
Let $X$ be a complex manifold of dimension $n$ and let $X^o\subset X$ be a Zariski open subset. Let $(E,h)$ be a holomorphic vector bundle on $X^o$ with a singular Hermitian metric such that $\Theta_h(E)\geq_{\rm Nak}^s 0$ (Definition \ref{defn_singularNakanopositive}). 
Let $\alpha=\alpha'\otimes\omega$ be an $E$-valued $(n,0)$-form, where $\alpha'$ is a section of $E$ and $\omega$ is an $(n,0)$-form. We use the notation $$\{\alpha,\alpha\}_h:=c_n|\alpha'|^2_h\omega\wedge\overline{\omega},\quad c_n = 2^{-n}(-1)^{\frac{n^2}{2}}.$$ The $L^2$-norm of $\alpha\in H^0(X^o,\omega_{X^o}\otimes E)$ is defined as
$$\|\alpha\|_h^2 = \int_{X^o} \{\alpha,\alpha\} \in [0,+\infty].$$
We follow \cite{HPS2018} to use the scaling $c_n$ in the $L^2$-norm.

Suppose $f:X\rightarrow B$ is a projective map to the open unit ball $B\subset \bC^r$, where $0\in B$ is a regular value of $f$. Then the central fiber $X_0=f^{-1}(0)$ is a projective manifold of dimension $n-r$. We assume that $X_0\cap X^o\neq\emptyset$ and denote by $(E_0,h_0)$ the restriction of $(E,h)$ to $X_0\cap X^o$. The $L^2$ extension theorem for vector bundles equipped with a singular Hermitian metric, originally developed by Ohsawa-Takegoshi \cite{OT1987}, has been further elaborated by Guan-Zhou \cite{GZ2015} and Guan-Mi-Yuan \cite{GMY2023}. This theorem plays a crucial role in the proof of the main theorem of this article. For more results on this direction, we refer the readers to \cites{Blocki2013,CPB2024,BP2008,Berndtsson1996,Demailly2000,DHP2013,OT1988} and the references therein.
\begin{thm}\cite{GMY2023}\label{thm_OTextension}
	Notations as above. Suppose that $\Theta_h(E)\geq_{\rm Nak}^s 0$ and $h_0\not\equiv +\infty$. Then for every $\alpha\in H^0(X_0\cap X^o,\omega_{X_0\cap X^o}\otimes E|_{X_0\cap X^o})$ with $\|\alpha\|_{h_0}^2<\infty$, there exists $\beta\in H^0(X^o,\omega_{X^o}\otimes E)$ with 
	\begin{align}
		\beta|_{X_0}=\alpha\wedge df\quad \textrm{and}\quad \|\beta\|_h^2\leq \mu(B)\cdot \|\alpha\|_{h_0}^2.
	\end{align}
\end{thm}
The minimal extension property for singular Hermitian metrics, which is closely related to the Ohsawa-Takegoshi extension theorem, was introduced by Hacon, Popa, and Schnell \cite{HPS2018}. This property allows for the extension of sections across a bad locus while maintaining control over the norm of the section.

	Let us still denote by $B\subset \bC^n$ the open unit ball.
	\begin{defn}[minimal extension property]
		A singular Hermitian metric $h$ on a torsion-free coherent sheaf $\sF$ is said to have the \emph{minimal
			extension property} if there exists a nowhere dense closed analytic subset $Z$ 
		with the following two properties:
		\begin{itemize}
			\item $\sF$ is locally free on $X\setminus Z$.
			\item For every embedding : $\iota:B\rightarrow X$ with $x =\iota(0)\in X\setminus Z$, and every $v\in  E_x$
			with $|v_{h,x}| = 1$, there is a holomorphic section $s\in H^0(B,\sF)$ such that
			$s(0) = v$ and $$\frac{1}{\mu(B)}\int_B|s|_h^2d\mu\leq 1,$$
			where $(E,h)$ denotes the restriction to the open subset $X(\sF)$.
		\end{itemize}
	\end{defn}
    According to \cite{DNWZ2023}, the minimal extension property of a $C^2$ metric is equivalent to the Nakano semi-positivity of its curvature form.
\section{$S_X(E,h)$ and its basic properties}\label{subsection_SX}
Now, let's turn our attention to the main object of this paper, $S_X(E,h)$. This concept builds upon the same idea introduced in \cite{SZ2022}. The main difference here is that we allow for the metric $h$ to be singular.

Let $X$ be a complex space of dimension $n$ and $X^o\subset X_{\rm reg}$ a dense Zariski open subset of the regular locus $X_{\rm reg}$. Let $(E,h)$ be a vector bundle on $X^o$ with a singular Hermitian metric.
	\begin{defn}
		$S_X(E,h)$ is a sheaf defined as follows: for an open subset $U\subset X$, the space $S_X(E,h)(U)$ consists of holomorphic $E$-valued $(n,0)$-forms $\alpha$ on $U\cap X^o$ such that $\{\alpha,\alpha\}$ is locally integrable near every point of $U$.
		
		Let $f:X\to Y$ be a holomorphic morphism to a complex manifold $Y$. We define $S_{X/Y}(E,h)$ as $S_X(E,h)\otimes f^\ast \omega_Y^{-1}$.
	\end{defn}
If $X=X^o$ (in particular, $X$ is smooth) and $E$ is a holomorphic line bundle, then $S_X(E,h)\simeq \omega_X\otimes E\otimes \sI(h)$. The sheaf $S_X(E,h)$ is a torsion-free $\sO_X$-module with properties described in \cite{SZ2022}. The proofs of these properties are analogous and will be omitted here.
\begin{lem}
   If $U\subset X^o$ be a dense Zariski open subset, then $S_X(E,h)=S_X(E|_{U},h|_U)$.	
\end{lem}
\begin{prop}[Functoriality Property]\label{prop_L2ext_birational}
	Let $\pi:X'\to X$ be a proper holomorphic map between complex spaces which is biholomorphic over $X^o$. Then $$\pi_\ast S_{X'}(\pi^\ast E,\pi^\ast h)=S_X(E,h).$$
\end{prop}
\begin{lem}\label{lem_L2_tensor}
	Let $(F,h_F)$ be a Hermitian vector bundle on $X$ where $h_F$ is a smooth metric. Then 
	$$S_X(E\otimes F|_{X^o},hh_F)\simeq S_X(E,h)\otimes F.$$
\end{lem}
We generalize the tameness condition introduced in \cite{SZ2022} to include singular Hermitian metrics. The concept of "tameness" is inspired by the theory of degeneration of Hodge structures \cites{Schmid1973,Cattani_Kaplan_Schmid1986} and the theory of tame harmonic bundles \cite{Simpson1988,Simpson1990,Mochizuki20072,Mochizuki20071}.
\begin{defn}\label{defn_tame_Hermitian_bundle}
	Let $X$ be a complex space and $X^o\subset X_{\rm reg}$ a dense Zariski open subset. A  vector bundle $(E,h)$ on $X^o$ with a singular Hermitian metric is called \emph{tame} on $X$ if, for every point $x\in X$, there is an open neighborhood $U$ of $x$, a proper bimeromorphic morphism $\pi:\widetilde{U}\to U$ which is biholomorphic over $U\cap X^o$, and a vector bundle $Q$ endowed with a smooth metric $h_Q$ on $\widetilde{U}$ such that the following conditions hold.
	\begin{enumerate}
		\item $\pi^\ast E|_{\pi^{-1}(X^o\cap U)}\subset Q|_{\pi^{-1}(X^o\cap U)}$ as a subsheaf.
		\item There is a singular Hermitian metric $h'_Q$ on $Q|_{\pi^{-1}(X^o\cap U)}$ so that $h'_Q|_{\pi^\ast E}\sim \pi^\ast h$ on $\pi^{-1}(X^o\cap U)$ and
		\begin{align}\label{align_tame}
			(\sum_{i=1}^r\|\pi^\ast f_i\|^2)^ch_Q\lesssim h'_Q
		\end{align}
		for some $c\in\bR$. Here $\{f_1,\dots,f_r\}$ is an arbitrary set of local generators of the ideal sheaf defining $\widetilde{U}\backslash \pi^{-1}(X^o)\subset \widetilde{U}$.
	\end{enumerate}
\end{defn}
\begin{rmk}
	The  following are typical examples of tame Hermitian metrics.
		\begin{itemize}
			\item A continuous Hermitian metric.
			\item Any singular Hermitian metric of type $e^{-\varphi}h$ on a vector bundle is tame. Here $h$ is a smooth metric and $\varphi$ is a quasi-psh function.
			\item The Hodge metric of a variation of Hodge structure is tame at its boundary points. This is a consequence of the norm estimate for the Hodge metric, which was established by Schmid \cite{Schmid1973} and Cattani-Kaplan-Schmid \cite{Cattani_Kaplan_Schmid1986}.
			\item A tame harmonic metric on a harmonic bundle remains tame at boundary points. This conclusion is drawn from the norm estimate of the tame harmonic metric, as established by Simpson \cite{Simpson1990} and Mochizuki \cite{Mochizuki20072, Mochizuki20071}.
	\end{itemize}
\end{rmk}
\begin{prop}\label{prop_S_coherent}
	Assume that the holomorphic vector bundle $(E,h)$ with a singular Hermitian metric satisfies the following conditions.
	\begin{enumerate}
\item For every point $x\in X$ there is a neighborhood $U$ of $x$, a bounded $C^\infty$ function $\varphi$ on $U\cap X^o$ such that $\sqrt{-1}\Theta_{e^{-\varphi}h}(E)\geq_{\rm Nak}^s0$ holds on $U\cap X^o$.
\item The holomorphic vector bundle $(E,h)$ is tame on $X$.
\end{enumerate}
Then $S_X(E,h)$ is a coherent sheaf.
\end{prop}
Notice that Condition (1) implies that $(E,h)$ is tame on $X^o$.
\begin{proof}
		Since the problem is local, we assume that $X$ is a germ of complex space. By replacing $h$ by $e^{-\varphi}h$ for some smooth bounded function $\varphi$ (this does not alter $S_X(E,h)$) we may assume that $(E,h)$ is Nakano semi-positive.
	Let $\pi:\widetilde{X}\to X$ be a desingularization so that $\pi$ is biholomorphic over $X^o$ and $D:=\pi^{-1}(X\backslash X^o)$ is a simple normal crossing divisor. For the sake of convenience, we will consider $X^o\subset \widetilde{X}$ as a subset. Since $(E,h)$ is tame, we assume the existence of a Hermitian vector bundle $(Q,h_Q)$ on $\widetilde{X}$ such that $E$ is a subsheaf of $Q|_{X^o}$ and there exists an integer $m\in \bN$ satisfying 
	\begin{align}\label{align_tame_1}
		|z_1\cdots z_r|^{2m}h_Q\lesssim h_Q'
	\end{align}
	where $z_1,\cdots,z_n$ are local coordinates on $\widetilde{X}$ with respect to which $D=\{z_1\cdots z_r=0\}$, and where  $h_Q'$ denotes a singular Hermitian metric on $Q|_{X^o}$ such that  $h_Q'|_E\sim h$.
It follows from Proposition \ref{prop_L2ext_birational} that there is an isomorphism
	\begin{align*}
		S_X(E,h)\simeq \pi_\ast\left(S_{\widetilde{X}}(E,h)\right).
	\end{align*}
	Since $\pi$ is a proper map, it suffices to show that $S_{\widetilde{X}}(E,h)$ is a coherent sheaf on $\widetilde{X}$. Since the problem is local and $\widetilde{X}$ is smooth, we may assume that $\widetilde{X}\subset\bC^n$ is the unit ball, such that $D=\{z_1\cdots z_r=0\}$. Without loss of generality we assume that $Q$ admits a global holomorphic frame $\{e_1,\dots,e_{l}\}$ and $h_0$ is the trivial metric associated with this frame, i.e.,
	\begin{align}\label{align_orthogonal_frame}
		\langle e_i,e_j\rangle_{h_0}=\begin{cases}
			1, & i=j \\
			0, & i\neq j
		\end{cases}.
	\end{align}
	 Since $Q$ is coherent,  the space $\Gamma(\widetilde{X},S_{\widetilde{X}}(E,h))$ generates a coherent subsheaf $\sJ$ of $Q$ by strong Noetherian property for coherent sheaves. We have the inclusion 
	$\sJ\subset S_{\widetilde{X}}(E,h)$ by the construction. It remains to prove the converse. By Krull's theorem (\cite[Corollary 10.19]{Atiyah1969}), it suffices to show that
	\begin{align}\label{align_Krull}
		\sJ_x+S_{\widetilde{X}}(E,h)_x\cap m_{\widetilde{X},x}^{k+1}Q=S_{\widetilde{X}}(E,h)_x,\quad\forall k\geq0,\quad\forall x\in\widetilde{X}.
	\end{align}
	Let $\alpha\in S_{\widetilde{X}}(E,h)_x$ be defined in a precompact neighborhood $V$ of $x$. Choose a $C^\infty$ cut-off function $\lambda$ such that $\lambda\equiv1$ near $x$ and ${\rm supp}\lambda\subset V$. 
Let
\begin{align*}
	\psi_k(z):=2(n+k+rm)\log|z-x|+|z|^2
\end{align*}
and $h_{\psi_k}:=e^{-\psi_k}h$, where $|z|^2:=\sum_{i=1}^n|z_i|^2$.  Let $\omega_0:=\sqrt{-1}\ddbar|z|^2$. Then
\begin{align*}
	\sqrt{-1}\Theta_{h_{\psi_k}}(E)= \sqrt{-1}\ddbar\psi_k+\sqrt{-1}\Theta_{h}(E)\geq^s_{\rm Nak} \omega_0.
\end{align*}
Since ${\rm supp}(\lambda\alpha)\subset V$ and $\dbar(\lambda\alpha)=0$ near $x$, we know that
\begin{align*}
	\|\dbar(\lambda\alpha)\|^2_{\omega_0,h_{\psi_k}}\sim \|\dbar(\lambda\alpha)\|^2_{\omega_0,h}\leq\|\dbar\lambda\|^2_{L^\infty}\|\alpha\|^2_{\omega_0,h}+|\lambda|^2\|\dbar\alpha\|^2_{\omega_0,h}<\infty
\end{align*}
Since there is a complete K\"ahler metric on $X^o$ by \cite[Lemma 2.14]{SZ2022},  \cite[Proposition 4.1.1]{CataldoAndrea1998} (see also  \cite[Th\'eor\`eme 5.1]{Demailly1982})  gives a solution to the equation $\dbar\beta=\dbar(\lambda\alpha)$ so that
\begin{align}\label{align_norm_beta_psi}
	\|\beta\|^2_{\omega_0,h}\lesssim\int_{X^o}|\beta|^2_{\omega_0,h}|z-x|^{-2(n+k+rm)}{\rm vol}_{\omega_0}\lesssim \|\dbar(\lambda\alpha)\|^2_{\omega_0,h_{\psi_k}}<\infty.
\end{align}
Thus $\gamma=\beta-\lambda\alpha$ is holomorphic and $\gamma\in \Gamma(\widetilde{X},S_{\widetilde{X}}(E,h))$. 

Notice that $\dbar\beta=0$ near $x$. We may shrink $\widetilde{X}$ and assume that
$$\beta=\sum_{i=1}^l f_ie_idz_1\wedge\cdots\wedge dz_n$$
for some holomorphic functions $f_1,\dots,f_l\in\sO_{\widetilde{X}}(X^o)$.
Noticing that $h_Q\sim h_0$, we can deduce  from  (\ref{align_tame_1}), (\ref{align_orthogonal_frame}) and (\ref{align_norm_beta_psi}) that 
\begin{align*}
	&\sum_{i=1}^l\int_{X^o}|f_i|^2|z_1\cdots z_r|^{2m}|z-x|^{-2(n+k+rm)}{\rm vol}_{\omega_0}\\\nonumber
	=&\int_{X^o}|\beta|^2_{\omega_0,h_0}|z_1\cdots z_r|^{2m}|z-x|^{-2(n+k+rm)}{\rm vol}_{\omega_0}\\\nonumber
	\lesssim&\int_{X^o}|\beta|^2_{\omega_0,h}|z-x|^{-2(n+k+rm)}{\rm vol}_{\omega_0}<\infty.
\end{align*}
This implies that for every $i=1,\dots,l$, we have $z_1^m\cdots z_r^m f_i\in m_{\widetilde{X},x}^{k+1+rm}$ (\cite[Lemma 5.6]{Demailly2012}). Consequently, $\beta_x$ belongs to $m_{\widetilde{X},x}^{k+1}Q$, and we establish the validity of  (\ref{align_Krull}).
\end{proof}
\subsection{Example: parabolic Higgs bundle}
We use the notations in \S \ref{section_exp_Higgs}. Let $X$ be a smooth projective variety and $D$ a reduced simple normal crossing divisor on $X$. Consider $(H, \{{_E}H\}_{E\in{\rm Div}_D(X)}, \theta)$ as a locally abelian parabolic Higgs bundle on $(X,D)$, which is polystable with respect to an ample line bundle $A$ on $X$. Let $h$ be a tame harmonic metric on $H|_{X\backslash D}$, which is compatible with the parabolic structure. The existence of such a metric is ensured by Simpson \cite{Simpson1990} for algebraic curves and by Mochizuki \cite{Mochizuki2006} in higher dimensions. Let $\overline{\theta}$ be the adjoint of $\theta$, and $\partial$ be the unique $(1,0)$-connection such that $\partial + \dbar$ is compatible with $h$. Consequently, $(H|_{X\backslash D}\otimes_{\sO_{X\backslash D}}\sA^0_{X\backslash D}, \nabla:=\partial+\dbar+\theta+\overline{\theta})$ specifies a meromorphic flat connection that remains regular along the divisor $D$. 
 Let $\nabla=\nabla^{1,0}+\nabla^{0,1}$ be the decomposition with respect to the bi-degree. Notice that $\nabla^{1,0}=\partial+\theta$ and $\nabla^{0,1}=\dbar+\overline{\theta}$.

Let $K\subset H|_{X\backslash D}$ be a locally free subsheaf such that the following conditions hold:
\begin{itemize}
	\item $\nabla^{0,1}(K)=0$, i.e. $K$ is holomorphic with respect to both the complex structures $\dbar$ and $\nabla^{0,1}$.
	\item $(\nabla-\theta)(K)\subset K\otimes\sA^{1,0}_{X\backslash D}$.
\end{itemize}
Let $(F,h_F)$ be an arbitrary Nakano semi-positive Hermitian vector bundle on $X$. Let $L$ be a line bundle on $X$ such that $L\simeq_{\mathbb{R}}B+N$, where $B$ is a semi-positive $\mathbb{R}$-divisor and $N$ is an $\mathbb{R}$-divisor on $X$ which is supported on $D$. Let $\varphi_N$ be a weight function associated with $N$. By \cite[Lemma 3.8]{SZtwisted}, there is a singular Hermitian metric $h_L$ on $L$ such that the following conditions hold:
\begin{enumerate}
	\item $h_L$ is smooth over $X\backslash{\rm supp}(N)$.
	\item \begin{align}\label{align_L_0}
		\sqrt{-1}\Theta_{h_L}(L|_{X\backslash{\rm supp}(N)})=\sqrt{-1}\Theta_{h_B}(B)|_{X\backslash{\rm supp}(N)}\geq0;
	\end{align} 
	\item   \begin{align}\label{align_norm_est_L}
		|e|_{h_L}\sim\exp(-\varphi_N)
	\end{align} for  a local generator $e$ of $L$. 
\end{enumerate}
It can be shown that $hh_Fh_L$ has Nakano semi-positive curvature on $X\backslash D'$ and is tame on $X$ \cite[Lemma 3.10]{SZtwisted}. Moreover, one has the following.
\begin{thm}\cite[Corollary 3.13]{SZtwisted}
	$\omega_X\otimes (P{_{D-N,(2)}}(H)\cap j_\ast K)\otimes F\otimes L\simeq S_X(K\otimes F|_{X\backslash D'}\otimes L|_{X\backslash D'},hh_Fh_L)$. 
\end{thm}
Thus Theorem \ref{thm_main} implies Theorem \ref{thm_Higgs}.
\subsection{Example: multiplier $S$-sheaf}
We use the notation in \S \ref{section_exp_S_sheaf}. Note that Griffiths's curvature formula ensures that $(S(\bV),h_Q)$ is Nakano semi-positive (\cite[Theorem 2.3]{SC2021}, see also \cite[Lemma 7.18]{Schmid1973}). Thus $(S(\bV)\otimes F,e^{-\varphi}h_Qh)$ is Nakano semi-positive. The tameness of $(S(\bV),h_Q)$ follows from the theory of degeneration of Hodge structure (see \cite[Proposition 5.4]{SZ2022}). Therefore $(S(\bV)\otimes F,e^{-\varphi}h_Qh)$ is tame on $X$. By Lemma \ref{lem_L2_tensor} one has the following.
\begin{thm}
	$S(IC_X(\bV),\varphi)\otimes F\simeq S_X(S(\bV)\otimes F,e^{-\varphi}h_Qh)$.
\end{thm}
Thus Theorem \ref{thm_main} implies Theorem \ref{thm_multi_S}.
\section{proof of the main theorem}
The proof of the main theorem is strongly influenced by the one in \cite{HPS2018}. 
\subsection{Construction of the metric on its locally free part}
To define the singular Hermitian metric $H$ on $\sF=f_\ast (S_{X/Y}(E,h))$, we first construct the metric on some Zariski open subset $Y\setminus Z$. Then, we extend it over $Z$ using the $L^2$ extension theorem \ref{thm_OTextension}. The constructions and proofs follow a similar approach as \cite{HPS2018}. It is worth noting that the tameness condition allows the arguments in \cite{HPS2018} to apply to those degenerate $(E,h)$ as well.
Thanks to Proposition \ref{prop_L2ext_birational}, we can choose a resolution of singularity for $X$ and assume that $X$ is a complex manifold throughout the proof.
Notice that $S_{X/Y}(E,h)$ is a torsion-free coherent sheaf (Proposition \ref{prop_S_coherent}). We begin by selecting a closed, nowhere dense analytic subset $Z\subset Y$ that satisfies the following conditions:
\begin{enumerate}
	\item The morphism $f$ is submersive over $Y\setminus Z$.
	\item $X_y\cap X^o\neq \emptyset$ for every $y\in Y\setminus Z$.
	\item The sheaf $\sF$ is locally free on $Y\setminus Z$.
	\item $\sF$ has the base change property on $Y\setminus Z$, that is, the natural morphism
	$\sF_y\to H^0(X_y,S_{X/Y}(E,h)|_{X_y})$ is an isomorphism for every $y\in Y\backslash Z$.
\end{enumerate}
Consequently, when restricted to the open subset $Y\setminus Z$, the sheaf $\sF$ forms a holomorphic vector bundle $F$ with a rank of $r\geq 1$. 
Conditions (3) and (4) ensure that whenever $y\in Y\setminus Z$ and $X_y:=f^{-1}(y)$, we have $F_y=\sF|_y=f_\ast(S_{X/Y}(E,h))|_y=H^0(X_y,S_{X/Y}(E,h)|_{X_y})$.

Let $(E_y,h_y)$ denote the restriction of $(E,h)$ to $X_y\cap X^o$. Then Theorem \ref{thm_OTextension} implies the following lemma.	
	\begin{lem}
		For any $y\in Y\setminus Z$, we have the inclusion
		$$H^0(X_y,S_{X_y}(E_y,h_y))\subset F_y=H^0(X_y,S_{X/Y}(E,h)|_{X_y}).$$

	\end{lem}
\begin{proof}
	If $h_y \equiv+ \infty$, then $H^0(X_y,S_{X_y}(E_y,h_y))$ is trivial. Therefore, we only need to consider the case when $h_y$ is not identically equal to $+\infty$. A small neighborhood $U$ of $y$ can be chosen such that it is biholomorphic to the open unit ball $B\in \bC^r$, and $\omega_Y$ is trivial on it. Given 
	$\alpha\in H^0(X_y,S_{X_y}(E_y,h_y))$, Theorem \ref{thm_OTextension}  provides a section $\beta\in H^0(U,S_{X}(E,h))$  such that $\beta|_{X_y}=\alpha\wedge df$.
\end{proof}
For each $y\in Y\setminus Z$, a singular Hermitian metric on $F_y$ can be defined as 
$$|\alpha|_{H,y}^2=\int_{X_y}\{\alpha,\alpha\}_{h_y}\in [0,+\infty],$$
and it is finite on $H^0(X_y,S_{X_y}(E_y,h_y))$.
This definition is valid because $X_y\setminus X^o$ has a measure of zero, so it does not cause any issues for the integral.

In order to patch the singular Hermitian metric $|-|_{H,y}$ together on $Y\setminus Z$, we select a point $y\in Y\setminus Z$ and an open neighborhood $U\subset Y\setminus Z$ that is biholomorphic to the open unit ball $B\subset \mathbb{C}^r$. By pulling everything back to $U$, we can assume  that $Y = B$, $Z = \emptyset$, and $y = 0$. Let's denote by $t_1,\dots,t_r$ the standard coordinates on $B$. Then the canonical bundle $\omega_B$ is trivialized by the global section $dt_1\wedge \cdots \wedge dt_r$, and the volume form on $B$ is given by

$$
d\mu=c_r(dt_1 \wedge\cdots dt_r) (d\bar{t}_1\wedge\dots d{\bar{t}}_r).
$$

We fix a holomorphic section $s\in H^0(B,F)$, and denote by $\beta=s\wedge (dt_1\wedge\cdots\wedge dt_r)\in H^0(B,\omega_B\otimes F)\simeq  H^0(X,S_X(E,h))$ the corresponding holomorphic $n$-form on $X$ with coefficients in $E$. 
Given that $f : X \rightarrow B$ is smooth, Ehresmann's fibration theorem implies that $X$ is diffeomorphic to the product $B\times X_0$. After selecting a K\"ahler metric $\omega_0$ on $X_0$, we can express
\begin{align}
	|\beta|_h^2=G\cdot d\mu\wedge \frac{\omega_0^{n-r}}{(n-r)!}.
\end{align}
Since $h$ is an increasing limit of $C^2$ metrics near every point, the function $G:B \times X_0\rightarrow [0, +\infty)$ is both lower semi-continuous and locally integrable.

 At every point $y\in B$, we then have, by construction,
\begin{align}
	|s(y)|_{H,y}^2=\int_{X_0}G(y,-)\frac{\omega_0^{n-r}}{(n-r)!}.
\end{align}
According to Fubini's theorem, the function $y \mapsto |s(y)|_{H,y}$ is measurable. Furthermore, since $X_0$ is compact and $G$ is locally integrable, $|s(y)|_{H,y} < \infty$ for almost every $y \in B$. As $F$ is coherent, it is generated over $B$ by finite many global sections. Therefore, the singular Hermitian inner product $|-|_{H,y}$ is finite and positive definite for almost every $y \in B$, and thus for almost every $y \in Y \setminus Z$. The above discussion ensures that the Hermitian inner products satisfy the conditions in Definition \ref{defn_shmonvb}, making them a singular Hermitian metric on $F$ on $Y \setminus Z$.
\begin{prop}
	 On $Y\setminus Z$, the singular Hermitian inner products 
	$|-|_{H,y}$ determine a singular Hermitian metric on the holomorphic vector bundle $F$.
\end{prop}

	\subsection{Extend the metric to the whole $Y$}
	The aim of this subsection is to extend the singular Hermitian metric from $Y\setminus Z$ to the entire $Y$. To achieve this, we will examine the measurable function $\psi:=\log |g|_{H^\ast}:Y\setminus Z\rightarrow [-\infty,+\infty]$, where $H^\ast$ is the induced singular Hermitian metric on $F^\ast$ and $g\in H^0(Y,\sF^\ast)$. Our goal is to demonstrate that the function $\psi$ is a plurisubharmonic function on $Y\setminus Z$ and is locally bounded near every point of $Z$, which will allow us to extend it as a plurisubharmonic function across the entire $Y$ using the Riemann extension theorem for holomorphic functions \cite[Theorem I.5.24]{Demailly2012}. 
	
	First, let us restate the Ohsawa-Takegoshi theorem as given in Theorem \ref{thm_OTextension} in a form that is more suitable for our subsequent purposes.
	\begin{lem}\label{lem_OT}
		For every embedding $\iota: B\rightarrow Y$ from the unit ball $B\subset \bC^{\dim Y}$ with $y = \iota(0)\in Y\setminus Z$, and for
		every 
		$\alpha\in F_y$ with $|\alpha|_{H,y} =1$, there is a holomorphic section $s\in H^0(B,\iota^\ast \sF)$ with
		$s(0) =\alpha$ 
		and
		$$\frac{1}{\mu(B)}\int_B|s|_H^2d\mu\leq 1.$$
	\end{lem}
\begin{proof}
	After pulling everything back to $B$, we can assume that $Y = B$ and $y = 0$. By Theorem \ref{thm_OTextension}, since $|\alpha|_{H,0} = 1$, there exists an element $\beta \in H^0(X, S_X(E,h))$ such that $\beta|_{X_0} = \alpha \wedge df$ and $\|\beta\|_h^2 \leq \mu(B)$. We can trivialize the canonical bundle $\omega_B$ using $dt_1 \wedge \cdots \wedge dt_r$, which allows us to consider $\beta$ as a holomorphic section $s \in H^0(B, \iota^\ast\sF)$ with $s(0) = \alpha$. Additionally, $\frac{1}{\mu(B)}\int_B|s|_H^2d\mu \leq 1$, as $d\mu = c_r(dt_1 \wedge \cdots dt_r)\wedge (d\bar{t}_1 \wedge \cdots d\bar{t}_r)$.
\end{proof}
\begin{prop}\label{prop_boundedness}
Every point in $Y$ has an open neighborhood $U\subset Y$ such that
$\psi=\log |g|_{H^\ast}$ is bounded from above by a constant on $U\setminus Z$.
\end{prop}
\begin{proof}
Given an arbitrary point $x\in Y$, we select two small open neighborhoods $U\subset V\subset Y$ of $x$, where $\overline{V}$ is compact,  $\overline{U}\subset V$, and for every point $y\in U$, there is an embedding $\iota: B \rightarrow Y$ of the unit ball $B\in \bC^r$ with $\iota(0) = y$ and $\iota(B)\subset V$.
Now, we aim to prove the existence of a constant $C$ such that $\psi\leq C$ on $U\setminus Z$.

Let $y\in U\setminus Z$ be a fixed point. If $\psi(y)=-\infty$, then there is nothing to prove. However, assuming $\psi(y)\neq -\infty$, we can use the definition of the metric on the dual bundle to find a vector $\alpha$ in $F_y$ that satisfies $|\alpha|_{H,y} =1$ and $\psi(y) = \log |g(\alpha)|$. We choose an embedding $\iota: B \rightarrow Y$ such that $\iota(0) = y$ and $\iota(B)\subset V$. According to Lemma \ref{lem_OT}, there exists a holomorphic section $s\in H^0(V,\sF)$ with $s(0) =\alpha$ and
$$\frac{1}{\mu(B)}\int_V|s|_H^2d\mu\leq 1.$$ 
Therefore, we have $\psi(y) = \log|g(s)|_y$, and the desired upper bound can be obtained from Proposition \ref{prop_compactness} below.
\end{proof}
\begin{prop}\label{prop_compactness}
	Fix a constant $K \geq 0$, and consider the set
$$S_K =\left\{s\in H^0(V, \sF)\mid\int_V |s|_H^2d\mu\leq K\right\}.$$
Then 
\begin{enumerate}
	\item every sequence $\{s_k\}\in S_K$ has a subsequence that converges uniformly on compact subsets;
	\item there is a constant $C \geq 0$ such that, for every section $s \in S_K$, the holomorphic
	function $g(s)$ is uniformly bounded by $C$ on the compact set $\overline{U}$.
\end{enumerate}
\end{prop}
\begin{proof}
For each section $s\in S_K$, we  define
$$\beta = s \otimes(dt_1\wedge \cdots \wedge dt_r)\in H^0(V,\omega_Y\otimes \sF)=H^0(f^{-1}(V),S_X(E,h)),$$
the corresponding holomorphic section of $S_X(E,h)$.
It satisfies that
$\|\beta\|_h^2=\int_V|s|_H^2d\mu\leq K$.
Because $\overline{V}$ is compact and $f$ is proper, we can cover $f^{-1}(V)$ with a finite number of open sets $W$ that are biholomorphic to the open unit ball in $\bC^n$.

Let $\pi:\widetilde{W}\to W$ be a desingularization such that $\pi$ is biholomorphic over $W^o:=W\cap X^o$, and $D:=\pi^{-1}(W\backslash W^o)$ is a simple normal crossing divisor. 
For the sake of simplicity, we consider $W^o$ as a subset of $\widetilde{W}$. Since $(E,h)$ is tame, we assume the existence of a $C^\infty$ Hermitian vector bundle $(Q,h_Q)$ on $\widetilde{W}$ such that $E$ is a subsheaf of $Q|_{W^o}$. Additionally, there exists an $m\in \mathbb{N}$ that satisfies the inequality
\begin{align}\label{align_tame_2}
	|z_1\cdots z_r|^{2m}h_Q\lesssim h
\end{align}
where $z_1,\cdots,z_n$ are local coordinates on $\widetilde{W}$ and $D=\{z_1\cdots z_r=0\}$.

To continue, we choose a set of holomorphic local frames $s_1,\dots,s_r$ of $Q$ on some open subset $W'\subset \widetilde{W}$. Let $h_0$ be the trivial metric associated with these frames, i.e.,  $(s_i,s_j)_{h_0}=\delta_{ij}$. Using this, we can write $\beta= \sum_{1\leq i\leq r}b_is_i \otimes dz_1\wedge \cdots \wedge dz_n$, where $b_i$ are holomorphic functions on $W'\cap W^o$.
Notice that the two Hermitian metrics $h_0$ and $h_Q$ are quasi-isometric, i.e., there exists some positive constant $C_1$ such that $\frac{1}{C_1}h_Q\leq h_0\leq C_1 h_Q$.
This leads to the following inequality:
\begin{align}\label{align_L2_imply_log}
	&\int_{W'} \sum_{1\leq i\leq r}|b_i|^2|z_1\cdots z_r|^{2m}(dx_1\wedge dy_1)\wedge \cdots \wedge (dx_n\wedge dy_n)\\\nonumber
	&\leq  C_1\int_{W'} \sum_{1\leq i\leq r}|b_i|^2|z_1\cdots z_r|^{2m}|s_i|^2_{h_Q}(dx_1\wedge dy_1)\wedge \cdots \wedge (dx_n\wedge dy_n)\stackrel{(\ref{align_tame_2})}{\leq} C_2\int_W |\beta|_h^2\leq C_2K,
\end{align}
for some positive constant $C_2$. In particular, the functions $(z_1\cdots z_r)^mb_i$ are holomorphic on $W'$.

Let $\{s_n\}$ be an arbitrary sequence in $S_K$. We denote by $\beta_n$ the corresponding holomorphic sections in $S_X(E,h)$ on $W'$. We can write $\beta_n$ as $\beta_n= \sum_{1\leq i\leq r}b_{n,i}s_i \otimes dz_1\wedge \cdots \wedge dz_n$. According to \cite[Proposition 12.5]{HPS2018}, by replacing it with a subsequence, we can assume that the sequence $\{(z_1\cdots z_r)^mb_{n,i}\}_{n\in\bN}$ converges uniformly on compact subsets in $H^0(W',\sO_{W'})$.
Notice that the natural injective morphism $\sO_{W'}\rightarrow \sO_{W'}(mD)$ induces an injective continuous mapping $$\iota:=\times(z_1\cdots z_r)^{-m}:H^0(W',\sO_{W'})\rightarrow H^0(W',\sO_{W'}(mD))$$ between Fr\'echet spaces (\cite[Ch. VIII,\S A]{GR2009}). 
Therefore, the sequence $\{b_{n,i}=\iota((z_1\cdots z_r)^mb_{n,i})\}$  converges to $b_{\infty,i}\in H^0(W',\sO_{W'}(mD))$.

Notice that (\ref{align_L2_imply_log}) implies that $S_{W'}(\pi^\ast E,\pi^\ast h) \subset \omega_{W'}\otimes \oplus_{i=1}^r\sO_{W'}(mD)s_i$, we conclude that $H^0(W', S_{W'}(\pi^\ast E,\pi^\ast h))$ is a closed subspace of $H^0(W', \omega_{W'}\otimes \oplus_{i=1}^r\sO_{W'}(mD)s_i)$ (\cite[Proposition VIII. A.2]{GR2009}). Therefore, the argument on the previous paragraph shows that $\beta_n$ converges to $\beta_\infty= \sum_{1\leq i\leq r}b_{\infty,i}s_i \otimes dz_1\wedge \cdots \wedge dz_n\in H^0(W', S_{W'}(\pi^\ast E,\pi^\ast h))$.
Since we are dealing with finitely many open sets, the corresponding sequence $\{\beta_n\}$ of every sequence $\{s_n\}$ in $S_K$ has a subsequence that converges uniformly on compact subsets to some $\beta_\infty\in H^0(f^{-1}(V),S_X(E,h))$. Let $s_\infty\in H^0(V,\sF)$ be the unique section such that $$\beta_\infty=s_\infty\otimes (dt_1\wedge \cdots\wedge dt_r).$$
Based on \cite[Proposition VIII. A.2]{GR2009}, the sequence $s_k$ converges to $s_\infty$ in the Fr\'echet space topology on $H^0(V, \sF)$. Thus (1) is proved.

We will prove second claim of the lemma by contradiction. Let us assume that $g(s)$ is not uniformly bounded on the compact set $\overline{U}$ for all $s\in S_K$. This implies that there exists a sequence $s_0, s_1, s_2, \dots\in S_K$ such that the maximum value of $|g(s_k)|$ on the compact set $\overline{U}$ is at least $k$. 
It follows from (1) that, by passing to a subsequence if necessary, the sequence $s_k$ converges to  some section $s_\infty\in H^0(V, \sF)$. As the map $g:H^0(V,\sF) \rightarrow  H^0(V,\sO_Y)$ is a continuous map, the holomorphic functions $g(s_k)$ then converge uniformly on compact subsets to $g(s_\infty)$. Consequently, $|g(s_k)|$ must be uniformly bounded on $\overline{U}$, contradicting the previous assumption.

	\end{proof}
\begin{prop}\label{prop_usc}
	For every $g\in H^0(Y,\sF^\ast)$, the function $\psi= \log |g|_{H^\ast}$ is upper
	semi-continuous on $Y\setminus Z$.
\end{prop}
\begin{proof}
Given an arbitrary point $y\in Y\setminus Z$, we can assume that $Y = B$, $Z = \emptyset$, and $y = 0$ by choosing a sufficiently small open neighborhood of $y$. In this case, $g\in H^0(B,F^\ast)$, and we just need to show the upper semi-continuity of $\psi= \log|g|_{H^\ast}$ at the origin.  Equivalently, we need to show that
	\begin{align}\label{align_usc}
		\limsup_{k\rightarrow +\infty}\psi(y_k)\leq \psi(0)
	\end{align}
	for every sequence $\{y_k\}\in B$ converging to the origin. We may assume that $\psi(y_k)\neq -\infty$ for all $k\in \bN$, and that the sequence $\psi(y_k)$ converges. By the definition of the metric on the dual bundle, there is a holomorphic section $s_k\in H^0(B,F)$ for each $k\in \bN$, such that $\psi(y_k)=\log |g(s_k)|_{y_k}$. By Lemma \ref{lem_OT}, we can choose these sections such that $|s_k(y_k)|_{H, y_k} = 1$ and
	$$\int_B|s_k|^2_Hd\mu\leq K$$
	for some constant $K\geq 0$. If necessary, we can pass to a subsequence and let $s_k$ converge uniformly on compact subsets to some $s\in H^0(B,F)$ using Proposition \ref{prop_compactness}. Then, the holomorphic functions $g(s_k)$ uniformly converge on compact subsets to $g(s)$. To prove (\ref{align_usc}), we must show $\log |g(s(0))|\leq \psi(0)$. The definition of the dual metric $H^\ast$ implies that
	$$\psi=\log |g|_{H^\ast}\geq \log |g(s)|-\log |s|_H.$$
	Therefore, it is equivalent to proving that $|s(0)|_{H}\leq 1$.
According to the discussions in \S 3.1, there is a lower semi-continuous function $G_k: B \times X_0 \rightarrow [0,+\infty)$ associated to each $s_k$, such that 
$$
1 = |s_k(y_k)|_{H,y_k}^2 = \int_{X_0} G_k(y_k,-) \frac{\omega_0^{n-r}}{(n-r)!}.
$$

Similarly, the section $s$ determines a lower semi-continuous function $G: B \times X_0 \rightarrow [0,+\infty)$. By Fatou's lemma, we can deduce that 
\begin{align}
	|s(0)|_{H}^2=\int_{X_0} G(0,-) \frac{\omega_0^{n-r}}{(n-r)!}\leq \int_{X_0}\liminf_{k \rightarrow +\infty} G_k(y_k,-) \frac{\omega_0^{n-r}}{(n-r)!}\leq \liminf_{k \rightarrow +\infty}\int_{X_0} G_k(y_k,-) \frac{\omega_0^{n-r}}{(n-r)!}=1.
\end{align}

Therefore, the proof is finished.

\end{proof}


By referring to Proposition \ref{prop_boundedness} and Proposition \ref{prop_usc}, verifying that $\psi$ satisfies the mean value inequalities, as derivable from Lemma \ref{lem_OT}, is all that's required to show that $\psi$ extends to a plurisubharmonic function on $Y$.
\begin{prop}
	 For every holomorphic mapping $\gamma:\Delta\rightarrow Y\setminus Z$, the function $\psi= \log |g|_{H^\ast}$ satisfies the mean-value inequality
$$(\psi\circ\gamma)(0)\leq \frac{1}{\pi}\int_{\Delta}(\psi\circ \gamma)d\mu.$$
Here $\Delta$ denotes the unit disc in $\bC$.
\end{prop}
\begin{proof}
	Since the inequality holds when $h\equiv+\infty$, we can assume that $h$ is not identically equal to $+\infty$. As the mapping $f: X\rightarrow Y$ is a submersion over $Y\setminus Z$, we can simplify the problem by considering the case when $Y = \Delta$. If $\psi(0) =-\infty$, then the mean-value inequality is always true. Assuming that $\psi(0)\neq-\infty$, we can select an element $\alpha\in F_0$, where $|\alpha|_{H,0} =1$, such that $$\psi(0) = \log |g|_{H^\ast, 0} =\log |g(\alpha)|.$$
By Lemma \ref{lem_OT}, there exists a holomorphic section $s\in H^0(\Delta, F)$ with $s(0) = \alpha$ and $\frac{1}{\pi}\int_{\Delta}|s|_H^2d\mu\leq 1$. With the definition of the metric $H^\ast$ on the dual bundle, we can derive the inequality:
$$|g|_{H^\ast} \geq \frac{|g(s)|}{|s|_H}.$$
This inequality yields $2\psi \geq \log |g(s)|^2 - \log |s|^2_H$. Integrating both sides leads to:
$$\frac{1}{\pi}\int_{\Delta}2\psi d\mu \geq \frac{1}{\pi}\int_{\Delta}\log |g(s)|^2d\mu - \frac{1}{\pi}\int_{\Delta}\log |s|_H^2d\mu.$$
Because $\log |g(s)|^2$ follows the mean-value inequality, the first term on the right side is at least $\log |g(\alpha)|^2 = 2 \psi(0)$. Further, the function $x\mapsto -\log x$ is convex, and the function $|s|_H^2$ is integrable. Hence, we can bound the second term using Jensen's inequality as:
$$-\log\left(\frac{1}{\pi}\int_\Delta |s|_H^2d\mu\right)\geq -\log 1=0.$$ 
Combining both results, we we reach the conclusion:
$$\frac{1}{\pi}\int_{\Delta}\psi d\mu\geq \psi(0),$$
establishing the mean-value inequality.
\end{proof}
To summarize, the function $\psi$ is plurisubharmonic on $Y\setminus Z$ and is locally bounded on $Y$. Consequently, it it extends to a plurisubharmonic function on the entire space $Y$. Furthermore, the singular Hermitian metric $H$ extends to the torsion-free coherent sheaf $\mathcal{F}$, and the minimal extension property is a direct consequence of Lemma \ref{lem_OT}. With this, the proof of Theorem \ref{thm_main} is concluded.
			\bibliographystyle{plain}
			\bibliography{positivity}
			
		\end{document}